\documentclass[12pt]{article}

\usepackage{amsbsy,amsfonts,amsmath,amssymb,amsthm
,bbm,enumerate,euscript,graphicx,epsfig,psfrag}

\usepackage{amsthm}

\newtheorem{theo}{Theorem}

\newtheorem{lemm}{Lemma}

\newcommand{\lbl}{\label}

\newcommand{\Po}{{\cal P}}
\newcommand{\Q}{{\cal Q}}

\def\N{\mathbb{N}}

\def\R{\mathbb{R}}
\def\E{\mathbb{E}}

\def\Pr{\mathbb{P}}

\def\0{{\bf 0}}

\renewcommand{\E}{\mathbb E \,}

\newcommand{\eqco}{\setcounter{equation}{0}}
\newcommand{\thco}{\setcounter{theo}{0}}
\newcommand{\prco}{\setcounter{prop}{0}}
\newcommand{\laco}{\setcounter{lemm}{0}}
\newcommand{\coco}{\setcounter{coro}{0}}
\newcommand{\cjco}{\setcounter{conj}{0}}

\newcommand{\deco}{\setcounter{defn}{0}}
\newcommand{\allco}{\eqco  \thco \prco \laco \coco \cjco \deco}
\newcommand{\X}{{\cal X}}

\newcommand{\A}{{\cal A}}

\newcommand{\Y}{{\cal Y}}
\newcommand{\I}{{\cal I}}

\newcommand{\eps}{\varepsilon}
\def\bdm{\begin{displaymath}}
\newcommand{\edm}{\end{displaymath}}
\def\benu{\begin{enumerate}}
\def\eenu{\end{enumerate}}
\def\beqn{\begin{equation}}
\def\eeqn{\end{equation}}
\def\be{\begin{equation}}
\def\ee{\end{equation}}
\def\bea{\begin{eqnarray}}
\def\eea{\end{eqnarray}}
\newcommand{\bean}{\begin{eqnarray*}}
\newcommand{\eean}{\end{eqnarray*}}
\newcommand{\bear}{\begin{eqnarray}}
\newcommand{\eear}{\end{eqnarray}}

\renewcommand{\epsilon}{\varepsilon}

\def\R{\mathbb{R}}

\renewcommand{\P}{{\mathbb P}}

\def\A{{\cal A}}

\def\qed{\hfill\hbox{${\vcenter{\vbox{
    \hrule height 0.4pt\hbox{\vrule width 0.4pt height 6pt
    \kern5pt\vrule width 0.4pt}\hrule height 0.4pt}}}$}}

\def\la{{\lambda}}

\begin{document}
\title{\bf On the critical threshold for continuum AB percolation}

\author{
 David Dereudre$^{1}$
and
Mathew D. Penrose$^{2}$\\
{\normalsize{\em Universit\'e de Lille and  University of Bath}}
 }

\maketitle


 \footnotetext{ $~^1$ Laboratoire Paul Painlev\'e, Universit\'e de Lille, France: david.dereudre@univ-lille1.fr}
 \footnotetext{ $~^2$ Department of
Mathematical Sciences, University of Bath, Bath BA2 7AY, United
Kingdom: {\texttt m.d.penrose@bath.ac.uk} }




\begin{abstract}
Consider a bipartite random geometric graph on
the union of two independent homogeneous Poisson point processes
in $d$-space, with distance parameter $r$ and intensities
$\lambda,\mu$. For any $\lambda>0$ we consider the percolation threshold $\mu_c(\lambda)$ associated to the parameter $\mu$. Denoting by
 $\lambda_c:= \lambda_c(2r)$ the percolation threshold for the
 standard Poisson Boolean model with radii $r$,
 we  show the lower bound $\mu_c(\lambda)\ge c\log(c/(\lambda-\lambda_c))$ 
for any $\lambda>\lambda_c$ with $c>0$ a fixed constant. In particular,
 $\mu_c(\lambda)$ tends to infinity when $\lambda$ tends to $\lambda_c$
from above.

\end{abstract}


\section{Introduction and statement of results}
In the continuum AB percolation model,
particles of two types A and B are scattered 
randomly in  Euclidean space as 
two independent Poisson processes, and edges
are added between particles of opposite type
that are sufficiently close together. This
provides a continuum analogue to lattice AB percolation
which is discussed in e.g. \cite{Grimm}. 
The model was introduced by Iyer and Yogeshwaran \cite{IY}, where
motivation is discussed in detail;
the main motivation comes from wireless communications networks
with two types of transmitter.
As discussed in \cite{PenAB},  a complementary (but distinct) continuum
percolation model with two types of particle is
the {\em  secrecy random graph} \cite{HS,PW}.

To describe continuum AB percolation more precisely, we make some
definitions. Let $d \in \N$.
Given any two locally finite sets $\X,\Y \subset \R^d$,
and given $r >0$, 
let $G(\X,\Y,r)$ be the 
bipartite graph with vertex sets $\X$ and $\Y$, 
 and
with an undirected edge $\{X,Y\}$ included
for each $X \in  \X$ and $Y \in \Y$ with 
 $|X-Y| \leq r$, where $|\cdot|$ is the
Euclidean norm in $\R^d$. Also,
let $G(\X,r)$ be the graph with vertex set $\X$
and
with an undirected edge $\{X,X'\}$ included
for each $X,X' \in  \X$ 
 with $|X-X'| \leq r$.

For $\lambda,\mu>0$  let $\Po_\lambda$, $\Q_\mu$ be
independent homogeneous Poisson point
 processes in $\R^d$ of intensity $\lambda, \mu $
respectively,  where
 we view each point process as a random subset of $\R^d$.
We  are here concerned with the 
 bipartite graph 
 $G(\Po_\lambda , \Q_\mu,r)$.


Let $\I$ be the class of   graphs
having at least one infinite component.
By a version of the
Kolmogorov zero-one law,
given parameters $r,\lambda,\mu$ (and $d$), we have
 $\P[ G(\Po_\la, \Q_\mu,r ) \in \I] \in \{0,1\}$.
Provided $r, \lambda,$ and $\mu$ are sufficiently large, 
we have $\P[ G(\Po_\la, \Q_\mu,r ) \in \I] = 1$; see
\cite{IY}, or \cite{PenAB}.
Set
$$
\mu_c(r,\lambda): = \inf \{\mu:\P [ G(\Po_\lambda,\Q_\mu,r)  \in \I]
=1\},
$$
with the infimum of the empty set interpreted as $+\infty$.
Also, for  the more standard one-type continuum percolation
graph $G(\Po_\la,r)$,
define
$$
\lambda_c(r): = \inf \{\lambda:\P [ G(\Po_\la,r) \in \I]=1\}.
$$
By scaling (see Proposition 2.11 of \cite{MR})
$\lambda_c(2r) = r^{-d} \lambda_c(2)$. The value of
 $\lambda_c(2)$ 
 is not  known analytically, but  
 is well known to be finite for $d \geq 2$ \cite{Grimm,MR},
and explicit bounds 
 are provided in \cite{MR}.
Simulation studies indicate that $1- e^{-\pi \lambda_c(2)} \approx 0.67635$
for $d=2$ \cite{QZ} and
 $1 - e^{-(4\pi/3) \lambda_c(2)} \approx 0.28957$ for
$d=3$ \cite{LZ}.

Obviously if $G(\Po_\la,\Q_\mu,r)\in \I$ then also
$G(\Po_\la,2r) \in \I$,
 and hence 
 $\mu_c(r,\lambda) = \infty$ for 
$\lambda < \lambda_c(2r)$. In \cite{IY,PenAB}, it
is proved that $\mu_c(r,\lambda) < \infty $ for $\lambda > \lambda_c(2r)$.
Indeed,
%
with $\pi_d$ denoting the volume of the unit radius ball in $d$ dimensions
we have
from \cite{PenAB} that
\bea
\limsup_{\delta \downarrow 0}
\left(
\frac{
\mu_c(r,\lambda_c + \delta) }{
\delta^{-2d} |\log \delta| } \right)
 \leq \left(\frac{4 \lambda_c(2r)^2}{r}
\right)^d d^{3d}(d+1) \pi_d.
\label{upbd}
\eea
 It is also indicated in \cite{PenAB} how, for any
given $\lambda > \lambda_c(2r)$, one can
compute an explicit upper bound 
 for $\mu_c(r,\lambda)$.

As mentioned in \cite{PenAB}, it is of interest to
give complementary {\em lower} bounds for $\mu_c(r,\lambda)$.
In this note, we make some progress in this direction by
showing that there exists $c>0$ (depending on $d$) such that for all
 $\delta>0$ we have

\bea
\label{lowerbound}
\mu_c(\frac 12,\lambda_c(1)+\delta) \ge  c \log(c/\delta).
\eea
By scaling arguments, the previous lower bound is true for any radius
 $r>0$ (after changing the constant $c$), so in particular,
\bea
\label{maineq}
\lim_{\delta \downarrow 0}
\mu_c(r,\lambda_c(2r)+\delta) = + \infty.
\eea
Immediately we obtain 
\bea
\mu_c(r,\lambda_c(2r)) = + \infty.
\label{criteq}
\eea

We note from (\ref{upbd}) that if $\lambda > \lambda_c(2r)$,
we can find finite $\mu$ such that $G(\Po_\lambda,\Q_\mu,r) \in \I$
almost surely. If we were able to prove this under the weaker hypothesis
that $G(\Po_\lambda,2r) \in \I$ almost surely, then combining this with
(\ref{criteq}) we would have shown that in fact 
$G(\Po_{\lambda_c},2r) \notin \I$ almost surely, which would
 solve the classic open problem of proving non-percolation
at the critical point  (in any dimension) for this continuum percolation model.

We shall prove (\ref{lowerbound}) in the next section.
Our strategy of proof goes as follows. We deem all A-particles
having no B-particle nearby to be {\em useless}, since they cannot
be used in any percolating AB cluster. Given $\mu$, we use
  a version of the technique of
{\em enhancement} to show that there exists a value of $\lambda$
such that $\Po_\lambda$
 is supercritical for $A$-percolation (with distance parameter
$2r$) but becomes subcritical after removal of all the 
useless particles (a thinning process with only local dependence).   
The technique of enhancement has previously been applied
to one-type continuum percolation  in \cite{FPR0,FPR}, and 
further discussion of enhancement can be found there.


\section{Proof of the lower bound}
\lbl{secpf1}
\allco
%

This section is devoted to proving the following theorem

\begin{theo}
\label{mainthm}
There exists $c>0$ (depending on $d$) such that for all $\delta>0$,
we have (\ref{lowerbound}). In particular,
\bea
\lim_{\delta \downarrow 0} \mu_c( \frac{1}{2}, \lambda_c(1)+ \delta ) = 
+ \infty.
\label{muclim}
\eea
\end{theo}

Fix $\mu >0$. To prove (\ref{muclim}),  we need to show that
there exists $\lambda > \lambda_c(1)$ such that
$G(\Po_\lambda,\Q_\mu,1/2) \notin \I$ almost surely.
To obtain (\ref{lowerbound}) we need a suitable quantitative lower bound
for this $\lambda$ (or rather, for $\lambda - \lambda_c$) in terms of $\mu$.

We fix some $\lambda_0 > \lambda_c(1)$. Given a
 realization of $(\Po_{\lambda_0},\Q_\mu)$, let
us say that a point $x \in \R^d$ is {\em useless} 
 if no point  of $\Q_\mu$ lies within distance $1/2$ of $x$.
 Otherwise, let us say $x$ is {\em useful}.
We shall apply these notions mainly (but not always)
in the case where $x$ is itself a point of
 $\Po_{\lambda_0}$.
 

Given also $p,q \in [0,1]$ let
 $\Po_{\lambda_0,p,q}$ be a thinned
version of $\Po_{\lambda_0}$ where each useful
point is independently retained with probability $p$ and
each useless point is independently retained with probability $q$. In particular $\Po_{\lambda_0,p,p}$ has the same distribution 
as $\Po_{\lambda_0 p}$.

For $R>0$ let $B_R$ denote the Euclidean ball of radius $R$ 
centred at the origin.  Let $\theta(p,q)$ be the probability that there exists
an infinite component of $G(\Po_{\lambda_0,p,q},1)$ that
includes at least one vertex in $B_1$, and, for $n \in \N$, let
 $\theta_n(p,q)$ be the probability that there exists
a component of $G(\Po_{\lambda_0,p,q},1)$ that
includes at least one vertex in $B_1$ and
at least one vertex outside $B_n$.
Then for all $p,q$ we have $\theta(p,q)= \lim_{n \to \infty} \theta_n(p,q)$.

\begin{lemm}\label{lemmeRusso}
For any $x\in\R^d$, we denote by $A_{x,n,p,q}$ the event that $G(\Po_{\lambda_0,p,q} \cup \{x\},1)$ contains a path including vertices both in $B_1$ and in $B_n^c$, but
$G(\Po_{\lambda_0,p,q} ,1)$ does not,
 and $F_x$ is the event that the vertex at $x$ is
useful.

 Then for any $n\ge 1$ and $p,q \in (0,1)$,
$$
\frac{\partial \theta_n(p,q)}{\partial p} = \int_{B_{n+1}} \Pr[ A_{x,n,p,q} \cap F_x ]  \lambda_0 dx 
$$
and 
$$
\frac{\partial \theta_n(p,q)}{\partial q} =  \int_{B_{n+1}} \Pr[ A_{x,n,p,q} \cap F_x^c ] \lambda_0 dx,
$$
\end{lemm}

\begin{proof}
Adapting the proof of Lemma 1 in \cite{FPR0},
 we prove only the first identity since the second is obtained in
 exactly the same way. We denote by $\mathcal{F}$ the $\sigma$-algebra generated by $(\Po_{\lambda_0},\Q_\mu)$. In particular $\mathcal{F}$ does not contain any information on the thinning procedures for the useful and useless vertices. We denote by $A_n$ the event that there exists a path
in $G(\Po_{\lambda_0,p,q},1)$ from $B_1$ to outside $B_n$. Let us note that the distribution of $\Po_{\lambda_0,p,q}$, given the $\sigma$-algebra $\mathcal{F}$, consists of a collection of independent Bernoulli variables which
 indicate whether the vertices are retained or removed by the 
thinning procedure. Then, applying the standard coupling of
 Bernoulli variables and Russo's formula 
 (also attributed to Margulis but in fact dating back at least to 
\cite[eqn (5.2)]{EP}),
 we obtain for any $h\in (0,1-p]$ that

$$ 0\le \Pr(  \Po_{\lambda_0,p+h,q}\in A_n |\mathcal{F})-\Pr(  \Po_{\lambda_0,p,q}\in A_n |\mathcal{F}) \le h \#(\Po_{\lambda_0}\cap B_{n+1}) $$

and 

$$
 \lim_{h \downarrow 0}
 \frac 1h\Big(\Pr(  \Po_{\lambda_0,p+h,q}\in A_n |\mathcal{F})-\Pr(  \Po_{\lambda_0,p,q}\in A_n |\mathcal{F})\Big)=\E[N_{n,p,q}|\mathcal{F}],
$$
where $N_{n,p,q}$ is the number of useful vertices in $\Po_{\lambda_0}$
that are pivotal for the occurrence of $A_n$ (with the $(p,q)$-thinning
applied to all of the vertices of $\Po_{\lambda_0}$ except for
the one being counted as pivotal). Recall that
 a vertex $x$  in a configuration $\X$ is said to be pivotal for an
 increasing event $\A$ if $\X$ belongs to $\A$ whereas $\X\backslash\{x\}$ does not. 

By the dominated convergence theorem we obtain

\begin{eqnarray*}
 \frac{\partial^+ \theta_n(p,q)}{\partial p} & =& \lim_{h\downarrow 0}
 \frac 1h \E\Big[\Pr(  \Po_{\lambda_0,p+h,q}\in A_n |\mathcal{F})-\Pr(  \Po_{\lambda_0,p,q}\in A_n |\mathcal{F})\Big]\\
 & = &  \E\Big[\lim_{h\downarrow 0}
 \frac 1h\Big[\Pr(  \Po_{\lambda_0,p+h,q}\in A_n |\mathcal{F})-\Pr(  \Po_{\lambda_0,p,q}\in A_n |\mathcal{F})\Big]\Big]\\
 & =& \E [ \E[N_{n,p,q}|\mathcal{F}]]=\E [N_{n,p,q}].
\end{eqnarray*}
By the Mecke formula (see \cite{LP}) it follows that
\begin{eqnarray*}
\E[ N_{n,p,q}]& = &\E\left[ \sum_{x\in\Po_{\lambda_0}} 
{\bf 1}_{\{x 
\text{ is useful and  pivotal for } A_n\}}\right]\\
& = & \int_{B_{n+1}} \Pr[ A_{x,n,p,q} \cap F_x ]   \lambda_0 dx.
\end{eqnarray*}
One may argue similarly for the left derivative.
The lemma is proved. 
\end{proof}

Now we want to apply enhancement arguments as in  \cite{FPR0,FPR}. For this we need to control the ratio between $\frac{\partial \theta_n(p,q)}{\partial q} $ and $\frac{\partial \theta_n(p,q)}{\partial p}$. The crucial lemma is given here.

\begin{lemm}\label{lemmegeometric}
Given $\alpha>0$ and $\lambda_0>0$, there exists a constant $c>0$ such that for any $p\ge \alpha$, $q\ge \alpha$, $\mu >0$, $n\ge c^{-1}$ and $x\in B_{n+1}$,
 \bea
\frac{
\Pr[A_{x,n,p,q}\cap F_x^c ]}{  
\Pr[A_{x,n,p,q} \cap F_x]} 
\geq ce^{-\mu/c}.
\label{pivratio}
\eea   
\end{lemm}
The proof is based on geometrical arguments, and 
 is given at the end of the section.

\begin{proof}[Proof of Theorem \ref{mainthm}]
Set $\lambda_c := \lambda_c(1)$ and $p_c := \lambda_c/\lambda_0$.
Choose $\alpha := p_c/2$.
Then 
by Lemmas \ref{lemmeRusso} and \ref{lemmegeometric},
there exists $c >0$ such that for any 
$p,q \in (\alpha,1)$, any $\mu >0$,
 any $n\ge c^{-1}$ and any $x\in B_{n+1}$ we have
\bea
\frac{\partial \theta_n(p,q)}{\partial q} 
\geq  \left( c e^{-\mu/c} \right) 
\frac{\partial \theta_n(p,q)}{\partial p} .
\label{partialratio}
\eea
Let $\delta >0$ be  small enough so that
$p_c+ \delta <1$ and $p_c - \delta(1+ (2/c) e^{\mu/c}) \geq \alpha$,
 that is,
 $\delta(1+(2/c)e^{\mu/c})\le p_c/2$. 
There exists $c'>0$ such that the choice
\bea
 \delta=c'e^{-\mu/c'}
\label{deltamu}
\eea
is suitable, for all $\mu >0$. 
By the finite-increments formula on the
segment $[(p_c-\delta,p_c-\delta) ; 
(p_c+\delta,p_c-\delta(1+\frac 2c e^{\mu/c})]$ 
and inequality \eqref{partialratio},
 $$
\theta_n(p_c+ \delta,p_c-\delta(1+\frac 2c e^{\mu/c})) 
\leq \theta_n(p_c-\delta,p_c- \delta).
$$
By passing to the limit $n \to +\infty$ and noting that $\theta(p_c- \delta,p_c-\delta) =0$ we obtain 
 $$
\theta(p_c+\delta,0)\le \theta(p_c+ \delta,p_c-\delta(1+\frac 2c e^{\mu/c})) 
\leq \theta(p_c-\delta,p_c- \delta)=0,
$$
so $G(\Po_{\lambda_0,p_c + \delta,0},1 ) \notin \I $ almost surely
and hence  $G(\Po_{\lambda_c + \delta \lambda_0},\Q_\mu, \frac 12) \notin \I$
almost surely. 

In conclusion, setting $\lambda:=\lambda_c + \delta \lambda_0$ 
with $\delta$ given by (\ref{deltamu}), 
we have $G(\Po_{\lambda},\Q_\mu,\frac 12) \notin \I$ almost surely.
Hence 
$$
\mu_c(\frac 12, \lambda )  \geq \mu = c' \log \left( \frac{\lambda_0 c'}{ \lambda - \lambda_c} \right),
$$ 
so  the theorem is proved.
\end{proof}

It remains now to show Lemma \ref{lemmegeometric}. 
 Let us first give two geometrical lemmas.
 We denote by $B_k(x)$ the translated ball $x+B_k$ and for
 each point $a\in \R^d$ and 
subset $S\subset \R^d$, we denote by $d(a,S)$ the Euclidean distance from $a$ to $S$.

\begin{figure}[!ht]
\begin{center}
\psfrag{X}{\small$x$}
\psfrag{Y}{\small$y$}
\psfrag{R}{\small$R$}
\psfrag{Rr}{\small$R+r$}
\psfrag{Rmr}{\small$R-r$}
\psfrag{Xp}{\small$x'$}
\psfrag{Yp}{\small$y'$}
\psfrag{B1x}{\small$B_1(x)$}
\psfrag{Bdelta}{\small$B_\delta(x')$}
\includegraphics[width=12cm,height=5cm]{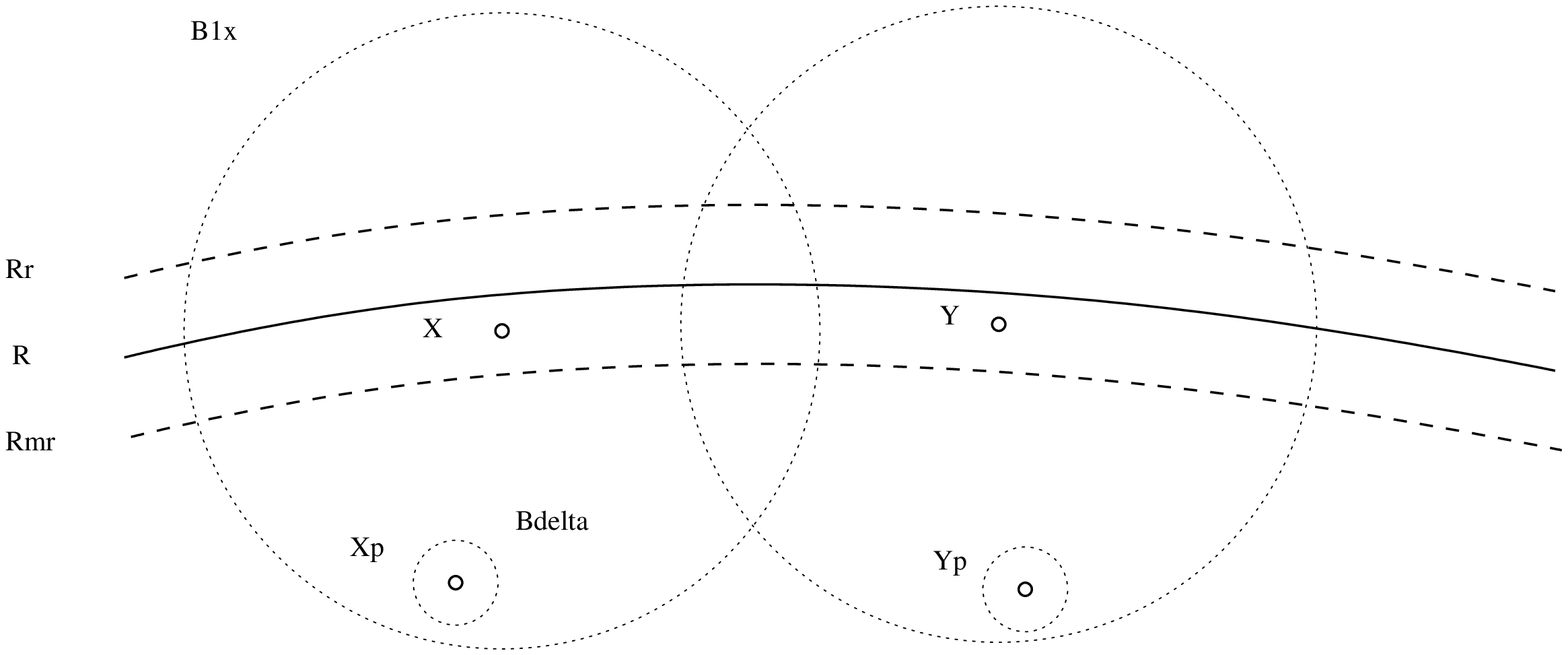}
\end{center}
\caption{An example of configuration $x$, $y$, $x'$ and $y'$ in Lemma \ref{lemmegeo1}.}
\end{figure}

\begin{lemm}\label{lemmegeo1}
There exists $K > 10$ such that
for all $R>K$ and  all  $r, \delta \in (0,1/K) $, we have:
 for all $x,y \in B_{R+r}(0)\backslash B_{R-r}(0)$ with $|x-y|>1$   there exist $x',y'\in B(0,R-1/2)$ satisfying 
\begin{itemize}
\item i) $|x-x'|\le 1-\delta$;
\item ii) $|y-y'|\le 1-\delta$;
\item iii) $|x'-y'|\ge 1+2\delta$;
\item iv) $d(x',B_R(0)^c\backslash B_1(x))\ge 1+\delta$;
\item v) $d(y',B_R(0)^c\backslash B_1(y))\ge 1+\delta$.
\end{itemize}
\end{lemm}

\begin{proof}
Fix $\eps \in (0,1)$.
Given $x,y \in B_{R+r}(0)\backslash B_{R-r}(0)$ with $|x-y|>1$,
 set 
\bean
x' :=x-(1-3\delta)x/|x|+2\delta (x-y)/|x-y|;
\\
 y' :=y-(1-3\delta)y/|y|+2\delta (y-x)/|x-y| . 
\eean
It is clear that for $R >10$ and $r$ and $\delta$ small enough  $x',y'\in B(0,R-1/2)$. Moreover items {\it i)} and {\it ii)} are trivially true by triangle inequality.
Also, provided $R$ is large enough we have $|x'-y'|\ge 1+3\delta $. 
Finally, provided $R$ is large enough and $r,\delta$ small enough,
 the distances $d(x',B_R(0)^c\backslash B_1(x))$ and 
$d(y',B_R(0)^c\backslash B_1(y))$ are bounded below by 
$\sqrt{2 - \eps}$. Therefore
choosing $R$ large enough and $r$, $\delta$ small enough, the
 last three items are satisfied.
\end{proof}

\begin{figure}[!ht]
\begin{center}
\psfrag{X}{\small$x$}
\psfrag{Y}{\small$y$}
\psfrag{R}{\small$R$}
\psfrag{Rr}{\small$R+r$}
\psfrag{Rmr}{\small$R-r$}
\psfrag{Xp}{\small$x'$}
\psfrag{Yp}{\small$y'$}
\psfrag{B1x}{\small$B_1(x)$}
\psfrag{Bdelta}{\small$B_\delta(y')$}
\includegraphics[width=12cm,height=6.3cm]{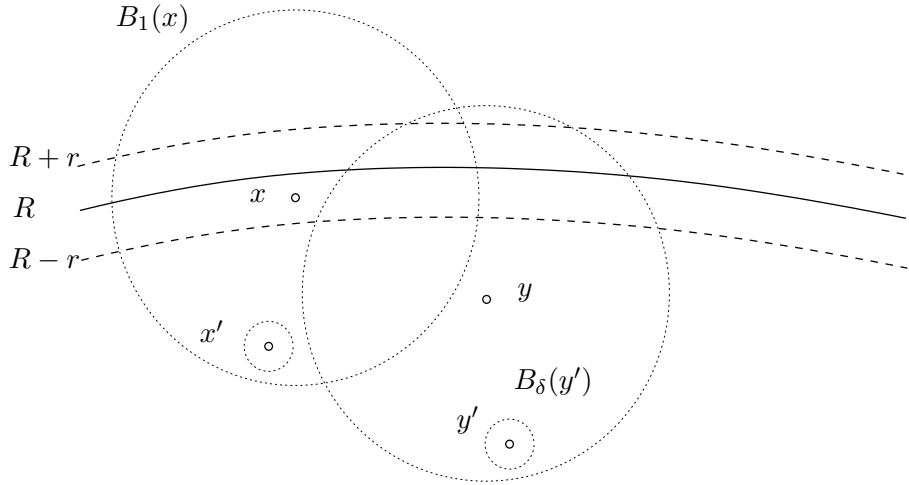}
\end{center}
\caption{An example of configuration $x$, $y$, $x'$ and $y'$ in Lemma \ref{lemmegeo2}.}
\end{figure}

\begin{lemm}\label{lemmegeo2}
Let $r>0$. Then there exists $K' >10$ such that for any $R>K'$
and any $\delta \in (0,1/K')$ we have: for all $x \in B_{R}(0)\backslash B_{R-r}(0)$ and $y \in B_{R-r}(0)$ with $|x-y|>1$ and $d(y,B_{R+r}(0)\backslash B_1(x))\le 1$, there exist $x',y'\in B(0,R-1/2)$ satisfying 
\begin{itemize}
\item i) $|x-x'|\le 1-\delta$;
\item ii) $|y-y'|\le 1-\delta$;
\item iii) $|x'-y'|\ge 1+2\delta$;
\item iv) $|x'-y|\ge 1+\delta$;
\item v) $|y'-x|\ge 1+\delta$;
\item vi) $d(y',B_R(0)^c\backslash B_1(y))\ge 1+\delta$;
\item vii) $d(x',B_{R+r}(0)^c\backslash B_1(x))\ge 1+\delta$.
\end{itemize}
\end{lemm}

\begin{proof}
Let $x \in B_{R}\backslash B_{R-r}$ and $y \in B_{R-r}$ with
 $|x-y|>1$ and $d(y,B_{R+r}\backslash B_1(x))\le 1$. Without loss of
 generality, we assume that $|x-y|<2$; otherwise, an accurate construction
 as in the previous lemma is possible. Let $\eps := 10^{-3}$
be fixed, and note that $\sin(\pi/3 - 3 \eps) >1/2$. 
Note also that provided $R$ is large enough, the vector
 $y-x$ is at an angle at most $\pi/3 + \eps$ with
the hyperplane tangent to $B_{|x|}$ at $x$. 

 Let $u$ and $v$ be unit vectors in the vector space generated by
 $x$ and $y$ such that 
the angle between $y-x$ and $u$ is equal to $\pi/3 +\eps$,
the angle between $y-x$ and $v$ is equal to $\pi/3 +2 \eps$,
 and such that both $u$ and $v$ have negative scalar product with $x$.
 (The choice of $u$ and of $v$ is unique provided that $R$ is large enough.)
 Then we set 
$x'=x+(1-3\delta)v$ and
 $y'=y+(1-3\delta)u$.

It is clear that for $\delta$ small enough  $x',y'\in B(0,R-1/2)$.
 Moreover items {\it i)} and {\it ii)} 
 are trivially true, 
and {\it iii)} holds provided $\delta$ is small enough.
The quadrilateral $x,y,y',x'$ is almost  a parallelogram;
the opposite edges $(y,y')$ and $(x,x')$ are  
the same length, and at an angle of $\eps$ to each other.
The angle between $y-x$ and $x'-x$ is equal to 
$\pi/3 + 2 \eps$, and all sides of this quadrilateral are
of length at least $1- 3 \delta$.
Provided $\delta$ is
small enough, the norms of the diagonals 
$|y'-x|$ and $|y-x'|$ exceed $1+ \delta$; that is,
items
 {\it iv)} and {\it v)} hold.
Also, for  $R$ sufficiently large, the angle between $y-y'$ and
 $y$ is smaller than $\pi/6 + 2 \eps $.
 Moreover $y\in B_{R-r}(0)$ so $d(y',B_R(0)^c\backslash B_1(y))\ge 1+r/2$.
 Item {\it vi)} follows for $\delta$ small enough. The proof of the last item  {\it vii)} is similar.
%
%
%
%
\end{proof}

\begin{proof}[Proof of Lemma \ref{lemmegeometric}]
Assume now that parameters $R>10$, $r>0$ and $\delta >0$
 are chosen such that all items in Lemmas \ref{lemmegeo1} 
 and \ref{lemmegeo2} are satisfied. 
Assume also that $R \in \N$ and that $\delta < R/99$.
 Noting that
$$
\frac{
\Pr[A_{x,n,p,q}\cap F_x^c ]}{  
\Pr[A_{x,n,p,q} \cap F_x]} 
=\frac{ \Pr[ F_x^c | A_{x,n,p,q}]}{ \Pr[ F_x | A_{x,n,p,q}]}\geq    \Pr[ F_x^c | A_{x,n,p,q}],
$$   
we have to find a lower bound for  $\Pr[ F_x^c | A_{x,n,p,q}]$ of the
 type $ce^{-\mu/c}$. Assume for now that $B_{R}(x) \subset
B_n \setminus B_1$. (We shall consider the other cases at the end.)

 Consider creating the Poisson processes $\Po_{\lambda_0}$, $\Q_\mu$ and $\Po_{\lambda_0,p,q}$
in stages, as follows
 (this is similar to arguments seen in \cite{FPR0} and \cite{FPR}).

 First we generate the points of $\Po_{\lambda_0}$ outside $B_{R}(x)$,
the points of
$\Q_{\mu}$  outside $B_{1/2}(x)$ and the retained points
 $\Po_{\lambda_0,p,q}$ outside $B_{R}(x)$.

At this stage, we let  $V$ be the set of vertices
of $\Po_{\lambda_0,p,q}$ created so far  which are 
connected by a path (in $G(\Po_{\lambda_0,p,q} \setminus B_R(x),1)$) to $B_1$,
and let $T$ be the set of vertices
of $\Po_{\lambda_0,p,q}$ created so far  which 
are connected by a path to $B_n^c$.    Let $V_1$ and $T_1$
denote the $1$-neighbourhood of $V$, $T$ respectively. 
Let $E_1$ be the event that
$V\cap T_1=\emptyset$, $T\cap V_1=\emptyset$ and $T_1\cap B_R\neq \emptyset$
and  $V_1\cap B_R\neq \emptyset$. 
If $A_{x,n,p,q}$ is realized, $x$ is pivotal and therefore 
$E_1$ occurs.

Now, assuming $E_1$ occurs, build up the point process of retained vertices
  $\Po_{\lambda_0,p,q}$ inwards into    $B_{R}(x) \cap (V_1 \cup T_1) $ 
from the boundary of the ball $B_{R}(x)$, until the appearance
of  the first new  vertex. 
 Denote this new vertex $X$. 
Let $E_2$ be the event that such a vertex exists,
that is, $E_2 := 
\{\Po_{\lambda_0,p,q} \cap B_R(x) \cap (V_1 \cup T_1) \neq \emptyset\}$.
If $A_{x,n,p,q}$ is realized, then $E_2$ must occur.

Assuming that $E_2$ occurs, we suppose $X \in T_1$
 (the other case, $X\in V_1$, can be treated in the same way).
We now have to distinguish several cases.

{\it Case 1:} $X\in B_R(x)\backslash B_{R-r}(x)$, which means that $X$ is close to the boundary of $B_R(x)$. Now we have also two sub-cases to consider:

{\it Case 1.1:} there exists $Y$ in $V\cap  B_{R+r}(x)\backslash B_{R}(x)$. Note that $X$ and $Y$ are exactly in position for applying Lemma \ref{lemmegeo1}. 
By that result, there exist $x_1$ and $y_1$ in $B_{R-1/2}(x)$ such that any point in $B_\delta(x_1)$ (respectively $B_\delta(y_1)$) can be the next point for the path from $B_n$ (respectively $B_1$) going to $x$. These points are now sufficiently far from the boundary of $B_R(x)$ to  create two paths $(x_i)_{1\le i\le R+2}$ and $(y_i)_{1\le i\le R+2}$ of $R+2$ points (deterministic, given
what has been revealed so far) from $T$ to $x$ and
 $V$ to $x$, where each step size in each path is at most $1-2 \delta$
(that is, $|x_{i+1} -x_{i} |< 1 - 2 \delta$ and
 $|y_{i+1} -y_{i} |< 1 - 2 \delta$ for each $i$, and
 $|x-x_{R+2}| \leq 1 - 2 \delta$ and
 $|x-y_{R+2}| \leq 1 - 2 \delta$),   and
   such that these paths remain a distance greater than $(1+2\delta)$
  from each other; see Figure \ref{figure}.

\begin{figure}[!ht]
\begin{center}
\psfrag{X}{\small$X$}
\psfrag{Y}{\small$Y$}
\psfrag{X1}{\small$x_1$}
\psfrag{Y1}{\small$y_1$}
\psfrag{X2}{\small$x_2$}
\psfrag{Y2}{\small$y_2$}
\psfrag{X3}{\small$x_3$}
\psfrag{Y3}{\small$y_3$}
\psfrag{X4}{\small$x_4$}
\psfrag{Y4}{\small$y_4$}
\psfrag{R}{\small$R$}
\psfrag{Rr}{\small$R+r$}
\psfrag{Rmr}{\small$R-r$}
\includegraphics[width=10cm,height=9cm]{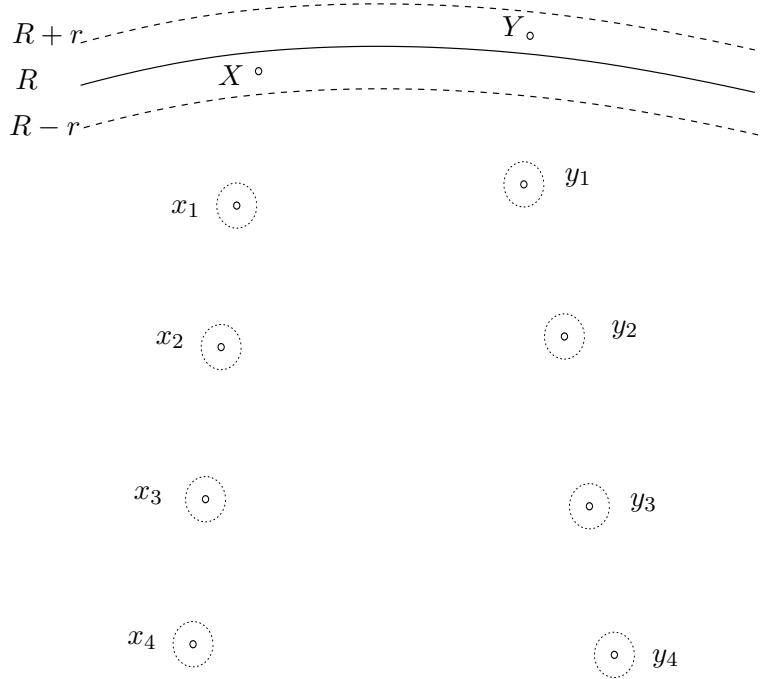}
\end{center}
\caption{Example of paths $(x_i)_{1\le i\le R+2}$ and $(y_i)_{1\le i\le R+2}$.}\label{figure}
\end{figure}

If each ball $B_{\delta}(x_i)$ or $B_{\delta}(y_i)$, 
$1\le i \le R+2$,  contains a single (retained) point,
and there are no other points of $\Po_{\lambda_0} \cap B_{R}(x)$ besides
those already considered, and there are no points of $\Q_{\mu}$ inside
$B_{1/2}(x)$,  then $A_{x,n,p,q} \cap F_x^c$ occurs. Recall that $p$ and $q$ 
are at least $\alpha$. Then the probability that all of the above
 occur is clearly bounded from below by $ce^{-\pi_d\mu/2^d}$ for a suitable constant $c>0$
(depending on $\lambda_0,R,r,\delta$ and $\alpha$),
 where $\pi_d$ is the volume of the unit ball in $\R^d$. 
 
 {\it Case 1.2:} the set $V\cap  B_{R+r}(x)\backslash B_{R}(x)$ is empty. In this case
we continue to build inwards the point process of retained vertices
  $\Po_{\lambda_0,p,q}$, but only in $V_1$.
Continue until the next new vertex appears, denoted $Y$   and then stop.
Let $E_3$ be the event that such a new vertex $Y$ does indeed appear.
If this occurs, then two sub-cases are possible:

{\it Case 1.2.1:} the point $Y$ is close to the boundary (i.e. $Y\in B_R(x)\backslash B_{R-r}(x)$). Then 
  $X$ and $Y$ are exactly in position for applying Lemma \ref{lemmegeo1} as in
 Case 1.1 and we conclude this case in the same way.

{\it Case 1.2.2:} the point $Y$ is in $B_{R-r}(x)$. Then $X$ and $Y$ are exactly in position for applying Lemma \ref{lemmegeo2}. So there exist $x_1$ and $y_1$ in $B_{R-1/2}(x)$ such that any points in $B_\delta(x_1)$ (respectively $B_\delta(y_1)$) can be the next point for the path from $B_n$ (respectively $B_1$) going to $x$. These points are now sufficiently far from the boundary of $B_R(x)$ to  create two paths $(x_i)_{1\le i\le R+2}$ and $(y_i)_{1\le i\le R+2}$ as in 
  Case 1.1.

{\it Case 2:} $X\in B_{R-r}(x)$, which means that $X$ is far from the boundary of $B_R(x)$. Then we continue to build inwards the point process of retained vertices
  $\Po_{\lambda_0,p,q}$, but only in $V_1$, as in the  Case 1.2.
 Continue until the next new vertex appears, denoted $Y$;  let
$E_4$ be the event that such a vertex $Y$ does appear, and assume $E_4$ occurs.
 Now $X$ and $Y$ are sufficiently far from the boundary (both in
 $B_{R-r}(x)$) in order to build $(x_i)_{1\le i\le R+2}$ and
 $(y_i)_{1\le i\le R+2}$ as in   Case 1.1.

Let $E$ be the event that events $E_1,E_2$, and (if in Case 1.2) $E_3$,
and (if in Case 2) $E_4$ (along with corresponding events for the
case when $X \in V_1$) all occur. If $A_{x,n,p,q}$ occurs,
 then $E$ must occur, and therefore
$$
\Pr [ F_x^c | A_{x,n,p,q} ] = \frac{ \Pr [ F_x^c \cap A_{x,n,p,q} \cap E] }{
\Pr[A_{x,n,p,q}] } 
\geq \Pr[ F_x^c  \cap A_{x,n,p,q} | E].
$$
  Thus, the conclusion of all these cases is that there exists a constant $c>0$ such that
  $$ 
\Pr[ F_x^c | A_{x,n,p,q}]\ge ce^{-\mu/c}.
$$

  Recall that we have been assuming $B_{R}(x) \subset B_n \setminus B_1$.
 A similar proof can be derived in the other two cases,
namely the case with $B_{R}(x)\cap B_1\neq \emptyset$ and the case
with $B_{R}(x)\cap B_n^c\neq \emptyset$.
We may assume $n \geq 99R$.
When $B_{R}(x)\cap B_1\neq \emptyset$,
 consider a path  $(x_i)_{1\le i\le 2R+2}$ from outside $B_{2R}(x)$ 
to $x$ and a path  $(y_i)_{1\le i\le R+2}$ inside $B_{2R}(x)$ joining 
$x$ to $B_1$.
When $B_{R}(x)\cap B_n^c\neq \emptyset$,
 consider a path  $(x_i)_{1\le i\le 2R+2}$ from outside $B_{2R}(x)$ 
to $x$ and a path  $(y_i)_{1\le i\le R+2}$ inside $B_{2R}(x)$ joining 
$x$ to $B_n^c$. For brevity we omit the details of these cases
here (which are similar to the corresponding cases treated in
\cite{FPR0}). Lemma \ref{lemmegeometric} is proved.
\end{proof}

{\it Acknowledgement:} This work was supported in part by the Labex CEMPI (ANR-11-LABX-0007-01), the CNRS GdR 3477 GeoSto and the ANR project PPP (ANR-16-CE40-0016).


\begin{thebibliography}{}





\bibitem{EP}
Esary, J. D. and Proschan, F. (1963).
Coherent structures of non-identical components.
{\em Technometrics} {\bf 5}, 191--209.

\bibitem{FPR0}
Franceschetti, M., Penrose,  M. D., and  Rosoman, T. (2010).
Strict inequalities of critical probabilities on Gilbert's
 continuum percolation graph.
 	arXiv:1004.1596 

\bibitem{FPR}
Franceschetti, M., Penrose,  M. D., and  Rosoman, T. (2011).
Strict inequalities of critical values in continuum percolation
{\em J. Statist. Phys.} {\bf 142}, 460--486.


\bibitem{Grimm}
Grimmett, G. (1999).
{\em Percolation.}
Second edition.
 Springer-Verlag, Berlin.


\bibitem{IY}
Iyer,  S. K. and Yogeshwaran, D. (2012).
Percolation and connectivity in AB random geometric graphs. 
{\em Adv. in Appl. Probab.} {\bf 44}, 21-41. 

\bibitem{LP}
Last, G. and Penrose, M. (2017)
{\em Lectures on the Poisson Process}.
Cambridge University Press, Cambridge.

\bibitem{LZ}
 Lorenz, C. D. and  Ziff, R. M. (2000).
 Precise determination of the critical percolation threshold for the
 three dimensional ``Swiss cheese''
model using a growth algorithm. 
{\em J. Chem. Phys.} {\bf 114}, 3659-3661.


\bibitem{MR}  Meester, R. and Roy, R. (1996). {\em Continuum Percolation}.
Cambridge University Press, Cambridge.


\bibitem{PenAB}
Penrose, M.D. (2014) 
 Continuum AB percolation and AB random geometric graphs. 
{\em J. Appl. Probab.} {\bf 51A},
 333-344. 

\bibitem{PW}
Pinto, P.C. and Win, Z. (2012)
Percolation and connectivity in the intrinsically secure
communications graph. 
{\em IEEE Trans. Inform. Theory} {\bf 58}, 1716-1730.

\bibitem{QZ}
Quintanilla, J. A. and  Ziff, R. M. (2007). 
Asymmetry of percolation thresholds of fully penetrable disks with two
 different radii.  {\em Phys. Rev. E} {\bf 76},  051115 
[6 pages].

\bibitem{HS}
Sarkar, A. and Haenggi, M. (2013)
Percolation in the Secrecy Graph.
 {\em Discrete Appl. Math.} {\bf 161}, 2120-2132.

\end{thebibliography}
\end{document}